\documentclass{amsproc}

\usepackage{amssymb}

\usepackage{graphicx}
\usepackage{float}
\usepackage{mathtools}
\usepackage{tikz-cd}
\usepackage{hyperref}
\usepackage[nameinlink,capitalize]{cleveref}
\usepackage{hyphenat}
\usepackage{breakcites}

\DeclareMathOperator{\Tr}{Tr}
\DeclareMathOperator{\ord}{ord}
\DeclareMathOperator{\Spec}{Spec}
\newcommand{\An}{\mathrm{an}}
\newcommand{\triv}{\mathrm{triv}}
\newcommand{\Div}{\mathrm{div}}
\newcommand{\NA}{\mathrm{NA}}

\newcommand{\gf}{\mathrm{gf}}
\newcommand{\val}{\mathrm{val}}
\newcommand{\PSH}{\mathrm{PSH}}
\newcommand{\NS}{\mathrm{NS}}
\newcommand{\QPSH}{\mathrm{QPSH}}
\newcommand{\EC}{\mathrm{EC}}
\newcommand{\ddc}{\mathrm{dd}^{\mathrm{c}}}
\newcommand{\sups}{\operatorname*{sup*}}
\newcommand{\cn}{\stackrel{\sim}{\longrightarrow}}

\usepackage[cmtip,all]{xy}

\newtheorem{theorem}{Theorem}[section]
\newtheorem{lemma}[theorem]{Lemma}
\newtheorem{proposition}[theorem]{Proposition}
\newtheorem{corollary}[theorem]{Corollary}

\theoremstyle{definition}
\newtheorem{definition}[theorem]{Definition}

\theoremstyle{remark}
\newtheorem{remark}[theorem]{Remark}

\numberwithin{equation}{section}

\begin{document}

\title{Operations on transcendental non-Archimedean metrics}


\author{Mingchen Xia}
\address{Chalmers Tekniska Högskola, Göteborg, Sweden}
\curraddr{}
\email{xiamingchen2008@gmail.com}
\thanks{The author benefited from discussions with Sébastien Boucksom and Pietro Mesquita-Piccione. In particular, the possibility of obtaining \cref{thm:comp_traceop} was suggested by Sébastien. The author would like to thank Tam\'as Darvas and Rémi Reboulet for their comments on the draft and the referee for numerous valuable suggestions. Je ne remercie pas le gouvernement français, car son inefficacité proverbiale m'a offert une période éprouvante pendant la rédaction de cet article.
The author is supported by Knut och Alice Wallenbergs Stiftelse grant KAW 2021.0231.}


\date{\today}

\dedicatory{Dedicated to Bo Berndtsson.}

\begin{abstract}
We define the basic pluripotential-theoretic operations in terms of the transcendental theory of non-Archimedean metrics introduced in \cite{DXZ23}. In particular, we establish that the analogue of Boucksom--Jonsson's envelope conjecture holds in our theory.
\end{abstract}

\maketitle

\tableofcontents

\section{Introduction}
Let $X$ be a reduced irreducible unibranch compact K\"ahler space of dimension $n$ and $\theta$ be a closed real smooth $(1,1)$-form on $X$. Assume that $\PSH(X,\theta)\neq \emptyset$. The purpose of this paper is to study the pluripotential theory of non-Archimedean metrics in the cohomology class of $\theta$ following the approach in \cite{DXZ23}. 

\subsection{Background}
The non-Archimedean pluripotential theory has been developed by many authors for different purposes in the last three decades. The study of semipositive metrics on line bundles on non-Archimedean spaces dates back to the breakthrough of Zhang \cite{Zhang95}. In \cite{CL06}, Chambert-Loir initiated the study of the non-Archimedean Monge--Amp\`ere measures. The pluripotential theory on non-Archimedean curves was developed in detail in the thesis of Thuillier \cite{Thu}. All these developments together lead to the natural question: To what extent is there a pluripotential theory on Berkovich spaces similar to its complex counterpart? 

This question has drawn the attention of many authors. In particular, so far we know that there are two general approaches in the literature. 

The first approach is the local approach, initiated by Chambert-Loir--Ducros \cite{CLD12} and Gubler \cite{Gub16}. In this approach, one defines the notions of differential forms and currents on a Berkovich space based on local tropicalizations. Then one can define the Monge--Amp\`ere operator using the usual differential formula. The problem is that this approach only works for regular plurisubharmonic functions and completely fails when a singularity is presented. This is an essential obstacle for applications to complex geometry. 

The second approach was initiated by Boucksom--Favre--Jonsson in \cite{BFJ15, BFJ16}, following the original ideas of Zhang. In the trivially valued case, this theory was further developed by Boucksom--Jonsson \cite{BJ18b}.
In this approach, one first considers metrics obtained from semipositive models and then defines general global plurisubharmonic metrics using approximations. This approach allows singularities but does not work in the local situation. It does not yield a sheaf of plurisubharmonic metrics. Both approaches have their own drawbacks. We do not have a completely satisfactory non-Archimedean pluripotential theory yet.

\subsection{The transcendental theory}
The non-Archimedean pluripotential theory is closely related to complex geometry. We refer to \cite{BBJ21} and the references therein. The basic idea is that the degeneration in the space of K\"ahler potentials can be characterized using global plurisubharmonic metrics on the non-Archimedean space. This idea has been intensively studied in \cite{BBJ21, Li20, DX21, XiaPPT}. 

We observe that there is a strong obstacle in applying the non-Archimedean geometry to complex geometry: The envelope conjecture \cite[Conjecture~5.14]{BJ18b} of Boucksom--Jonsson remains open. This conjecture together with its various consequences is indispensable for various applications. 

In our previous work \cite{DXZ23}, we introduced a different framework for the non-Archimedean global pluripotential theory based on the series of works starting from \cite{BBJ21, RWN14}. Roughly speaking, we represent a non-Archimedean metric in $\PSH^{\NA}(X,\theta)$ by a \emph{test curve} $\Gamma=(\Gamma_{\tau})_{\tau}$, a concave curve of $\mathcal{I}$-model potentials. Due to the pathology caused by zero-mass, the rigorous definition requires a projective limit construction as well, we refer to \cref{def:PSHNA1} for the rigorous definition. In this short introduction, we will only talk about test curves to simplify our presentation.

In the algebraic setting, when Boucksom--Jonsson's theory is applicable, there is an obvious comparison map sending $\Gamma$ to a function on the non-Archimedean space. See \cref{def:napotentialdaleth} for the details. The main result in \cite{DXZ23} asserts that when $X$ is smooth, this comparison map is in fact a bijection to Boucksom--Jonsson's theory, hence giving a complex representation of the non-Archimedean pluripotential theory. When $X$ is only unibranch, Boucksom--Jonsson's theory is still applicable. In this case, we prove the following result:
\begin{theorem}[{=\cref{cor:BJconj}}]
    Assume that $X$ is a reduced irreducible unibranch projective variety over $\mathbb{C}$ and $L$ is a pseudoeffective line bundle on $X$. Then the following are equivalent:
    \begin{enumerate}
        \item Boucksom--Jonsson's envelope conjecture holds for $(X,L)$;
        \item The comparison map extends to a bijection from the space of non\hyp{}Archimedean potentials in the sense of \cite{DXZ23} to the corresponding space in the sense of \cite{BJ18b}.
    \end{enumerate}
\end{theorem}
Note that Boucksom--Jonsson's envelope conjecture is known to hold when $X$ is smooth, as confirmed in \cite{BJ22b}. A different proof was given in \cite{DXZ23}. 

Using the comparison map, we can effectively translate the pluripotential-theoretic operations in Boucksom--Jonsson's theory to our transcendental theory. It turns out that in most cases, these translations can be extended to more general situations where the comparison fails or when Boucksom--Jonsson's theory is not applicable. This is exactly what we do in this paper. The final result is summarized in \cref{tab:corr}. It can be regarded as a duality between non-Archimedean geometry and convex geometry.

\begin{table}[h]
\begin{tabular}{|l|l|}
\hline
\textbf{NA potentials} & \textbf{Test curves} \\ \hline
Addition by a constant            & Translation                                                \\ \hline
Maximum                           & Concave envelope of maximum                                \\ \hline
Decreasing limit         & Decreasing limit                                           \\ \hline
Increasing limit                  & Increasing limit                                           \\ \hline
Regularized supremum              & Increasing limit of concave envelopes of maxima            \\ \hline
Rescaling                         & Rescaling and reparametrization                            \\ \hline
Addition                          & Infimal involution                                         \\ \hline
Restriction to subvarieties       & Trace operator                                             \\ \hline
\end{tabular}
\caption{Correspondence non-Archimedean/convex}
\label{tab:corr}
\end{table}

Most results from this table are quite straightforward to experts. The only exception is perhaps the last line (=\cref{thm:comp_traceop}), where we establish a quantitative version of the fact that a non-Archimedean plurisubharmonic metric is determined by its restriction to divisorial points. We hope that this result, after suitable extension, could serve to give a proof of Boucksom--Jonsson's envelope conjecture in general.

As a particular case from this table, we have
\begin{theorem}[{=\cref{thm:enve}}]
    The transcendental analogue of the envelope conjecture holds in the theory of \cite{DXZ23}.
\end{theorem}
This result slightly generalizes \cite[Theorem~1.2]{DXZ23} in that the latter only concerns the regularized supremum of increasing nets while the former works for a general bounded from above family.

Given the fact that Boucksom--Jonsson's envelope conjecture is widely open in general, we expect that our theory is more useful on normal varieties. In particular, we hope that the non-Archimedean theory of K-stability could be effectively translated to our theory. The very first step in this direction is carried out in \cite{XiaPPT}. On the other hand, our theory works more generally for transcendental classes on compact K\"ahler manifolds. It should be useful to the transcendental applications as well, for example in the study of cscK metrics as in \cite{SD18} or probably in the study of the transcendental Morse inequality.

On the other hand, although we have developed a non-Archimedean pluripotential theory, it is not entirely clear how to build a non-Archimedean space on which the non-Archimedean plurisubharmonic functions live. This gap is fixed in the very recent work \cite{MP24} of Mesquita-Piccione.

Most results from the current paper have been announced in several conferences. The author hopes that the details presented in the current paper are of interest to general readers as well. We assume that the readers are familiar with the complex pluripotential theory and Berkovich geometry. 

\subsection{Conventions}
In the whole paper, the Monge--Amp\`ere type products are always taken in the non-pluripolar sense. We adopt the convention that $\ddc = \frac{\mathrm{i}}{2\pi}\partial\overline{\partial}$. A variety refers to a separated scheme of finite type over $\mathbb{C}$.

\section{Pluripotential theory on unibranch spaces}
Let $X$ be a reduced irreducible unibranch compact K\"ahler space of dimension $n$ and $\theta$ be a closed real $(1,1)$-form on $X$. Here and in the sequel, a  $(1,1)$-form is always assumed to be smooth.
Assume that $\PSH(X,\theta)\neq \emptyset$. There is a well-developed pluripotential theory in this setting. For the Bedford--Taylor theory, we refer to \cite{Dem85} and for the non-pluripolar theory, we refer to \cite{Xia19}.

Note that we have assumed $X$ to be irreducible. This is not a very severe restriction: Since $X$ is unibranch, its irreducible components are just its connected components. Every result in this paper extends in the obvious way to general reduced unibranch compact K\"ahler spaces.

\subsection{Divisorial valuations}\label{subsec:div}
\begin{definition}
A \emph{prime divisor} over a reduced irreducible compact complex space $Z$ is a connected smooth hypersurface $E\subseteq X'$, where $X'\rightarrow Z$ is a projective resolution.

Two prime divisors $E_1\subseteq X'_1$ and $E_2\subseteq X'_2$ over $Z$ are \emph{equivalent} if there is a common resolution $X''\rightarrow X$ dominating both $X_1'$ and $X_2'$ such that the strict transforms of $E_1$ and $E_2$ coincide. 

The set $Z^{\Div}$ is the set of pairs $(c,E)$, where $c\in \mathbb{Q}_{>0}$ and $E$ is an equivalence class of a prime divisor over $Z$. For simplicity, we will denote the pair $(c,E)$ by $c\ord_E$, although one should not really think of this object as a valuation unless $Z$ is projective and irreducible.
\end{definition}
Note that a prime divisor on $Z$ does not always define a prime divisor over $Z$ if $Z$ is singular.

\begin{definition}\label{def:napotentialdiv}
Given $\varphi\in \PSH(X,\theta)$, we define its \emph{associated non-Archimedean potential} as the map $\varphi^{\An}\colon X^{\Div}\rightarrow \mathbb{R}$ given by  
\[
\varphi^{\An}(c\ord_E)\coloneqq -c\nu(\varphi,E).
\]
Here associated with $E$, there is a projective resolution $\pi\colon Y\rightarrow X$ such that $E$ is a connected smooth hypersurface in $Y$. The notation $\nu(\varphi,E)$ means the generic Lelong number of $\pi^*\varphi$ along $E$. By Siu's semicontinuity theorem, this quantity is independent of the choice of $\pi$.
\end{definition}

\subsection{The pluripotential theory}

\begin{definition}
    Assume that $X$ is smooth. Given $\varphi\in \PSH(X,\theta)$, we define
    \[
    \begin{aligned}
    P_{\theta}[\varphi]=&\sup \left\{ \psi\in \PSH(X,\theta):\psi\leq 0,\int_X \theta_{\varphi}^n=\int_X \theta_{\psi}^n,\varphi\leq \psi+C \text{ for some }C\in \mathbb{R} \right\},\\
    P_{\theta}[\varphi]_{\mathcal{I}}=&\sup \left\{ \psi\in \PSH(X,\theta):\psi\leq 0, \mathcal{I}(k\varphi)=\mathcal{I}(k\psi) \text{ for all }k\in \mathbb{Z}_{>0} \right\}.
    \end{aligned}
    \]
    
    A potential $\varphi\in \PSH(X,\theta)$ is \emph{model} (resp. \emph{$\mathcal{I}$-model}) if $\varphi=P_{\theta}[\varphi]$ (resp. $\varphi=P_{\theta}[\varphi]_{\mathcal{I}}$).

    A potential $\varphi\in \PSH(X,\theta)$ is $\mathcal{I}$-good if $\int_X \theta_{\varphi}^n>0$ and $P_{\theta}[\varphi]=P_{\theta}[\varphi]_{\mathcal{I}}$.
\end{definition}
Here $\mathcal{I}(k\varphi)$ denotes the multiplier ideal sheaf of $k\varphi$.
For a detailed study of these notions, we refer to  \cite{DX21, DX22, Xia22}. 
\begin{definition}
    We say a potential $\varphi\in \PSH(X,\theta)$ is \emph{model} (resp. \emph{$\mathcal{I}$-model}, \emph{$\mathcal{I}$-good}) if for a projective resolution $\pi\colon Y\rightarrow X$, $\pi^*\varphi\in \PSH(Y,\pi^*\theta)$ is model (resp. $\mathcal{I}$-model, $\mathcal{I}$-good).
\end{definition}
All three notions are independent of the choice of $\pi$.

\begin{definition}
    Given $\varphi\in \PSH(X,\theta)$, we define $P_{\theta}[\varphi], P_{\theta}[\varphi]_{\mathcal{I}} \in \PSH(X,\theta)$ as follows: Let $\pi\colon Y\rightarrow X$ be a projective resolution, then we require the following holds:
    \[
    \pi^*P_{\theta}[\varphi]=P_{\pi^*\theta}[\pi^*\varphi],\quad \pi^*P_{\theta}[\varphi]_{\mathcal{I}}=P_{\pi^*\theta}[\pi^*\varphi]_{\mathcal{I}}.
    \]
    By Zariski's main theorem, both $P_{\theta}[\varphi]$ and $P_{\theta}[\varphi]_{\mathcal{I}}$ are uniquely determined. 
\end{definition}
Both operations are clearly independent of the choice of $\pi$.
\begin{definition}
    Two quasi-psh functions $\varphi$ and $\psi$ on $X$ are \emph{$P$-equivalent} if for some K\"ahler form $\omega$ on $X$ such that $\varphi,\psi\in \PSH(X,\omega)$ and both $\omega_{\varphi}$ and $\omega_{\psi}$ have positive masses, we have $P_{\omega}[\varphi]=P_{\omega}[\psi]$. We write $\varphi\sim_P \psi$ in this case.
\end{definition}
Here and in the sequel, when we say a potential has positive mass, we always refer to the non-pluripolar mass.

\begin{definition}
    Two quasi-psh functions $\varphi$ and $\psi$ on $X$ are \emph{$\mathcal{I}$-equivalent} if for some K\"ahler form $\omega$ on $X$ such that both $\varphi,\psi\in \PSH(X,\omega)$, we have $P_{\omega}[\varphi]_{\mathcal{I}}=P_{\omega}[\psi]_{\mathcal{I}}$. We write $\varphi\sim_{\mathcal{I}} \psi$ in this case.
\end{definition}
Both conditions are independent of the choice of $\omega$: For the $\mathcal{I}$-equivalence relation, this follows from \cite[Theorem~2.13]{DX22}. The statement for the $P$-equivalence relation is trickier. As we do not need it in the paper, we omit the proof.

\begin{proposition}\label{prop:decnetPproj}
Suppose that $\varphi^i\in \PSH(X,\theta)$ ($i\in I$) is a decreasing net of $\mathcal{I}$-model potentials and $\inf_i \varphi^i$ has positive mass.
Then for all semipositive  closed real $(1,1)$-form $\omega$ on $X$, we have
\begin{equation}\label{eq:pinfI}
P_{\theta+\omega}\left[\inf_{i\in I}\varphi^i\right]_{\mathcal{I}}=\inf_{i\in I}P_{\theta+\omega}[\varphi^i]_{\mathcal{I}}.
\end{equation}
\end{proposition}

The proof relies on the $d_S$-pseudometric introduced in \cite{DDNLmetric}.

\begin{proof}

We may assume that $X$ is smooth. In fact, we can always take a projection resolution of singularities $\pi:Y\rightarrow X$, then as explained in \cite[Section~3.5]{Xia19}, we have a canonical bijection $\PSH(X,\theta)\cong \PSH(Y,\pi^*\theta)$. Moreover, under this bijection, the pluripotential-theoretic operations on both spaces correspond to each other. So it suffices to prove \eqref{eq:pinfI} with $Y$, $\pi^*\theta$, $\pi^*\omega$ and $\pi^*\varphi^i$ in place of $X$, $\theta$, $\omega$ and $\varphi^i$. This reduces the problem to the case where $X$ is smooth. In the sequel, we will omit this kind of arguments.

Observe that both sides of \eqref{eq:pinfI} are $\mathcal{I}$-model in $\PSH(X,\theta+\omega)$, so it suffices to show that they have the same generic Lelong numbers. Fix a prime divisor $E$ over $X$.
As shown in \cite[Proposition~4.8]{DDNLmetric}\footnote{The result is stated for a sequence, but the proof works for a net as well.}, $P_{\theta+\omega}[\varphi^i]_{\mathcal{I}}\xrightarrow{d_S}\inf_{i\in I}P_{\theta+\omega}[\varphi^i]_{\mathcal{I}}$. Hence, by \cite[Theorem~0.10]{Xia22},
\[
\lim_{i\in I}\nu(\varphi^i,E)=\nu\left(\inf_{i\in I}P_{\theta+\omega}[\varphi^i]_{\mathcal{I}},E\right).
\]
By \cite[Proposition~4.8]{DDNLmetric} again, $\varphi^i\xrightarrow{d_S}\inf_{j\in I}\varphi^j$. It follows from \cite[Theorem~0.10]{Xia22} that
\[
\lim_{i\in I}\nu(\varphi^i,E)=\nu\left(\inf_{i\in I}\varphi^i,E\right).
\]
We conclude \eqref{eq:pinfI}.
\end{proof}

\begin{lemma}\label{lma:supIbon}
Suppose that $\varphi^i\in \PSH(X,\theta)$ ($i\in I$) is a non-empty family of $\mathcal{I}$-good potentials, uniformly bounded from above. Then $\sups_{i\in I} \varphi^i$ is also $\mathcal{I}$-good.
\end{lemma}
\begin{proof}
We may assume that $X$ is smooth.
Assume that this result is known when $I$ is finite.
By Choquet's lemma, we therefore assume that $I=\mathbb{Z}_{>0}$ and $\varphi^i$ is an increasing sequence. In this case, the result follows from \cite[Theorem~4.6]{Xia21}.

It remains to handle the case where $I$ has two elements. The argument can be found in \cite[Section~6.2]{SGPT}.
\end{proof}

\begin{lemma}\label{lma:Imodeldeclelong}Let $\omega$ be a semipositive closed real $(1,1)$-form on $X$.
Let $\varphi^j\in \PSH(X,\theta+\epsilon_i\omega)$ be a decreasing sequence of model potentials, where $\epsilon_j$ is a decreasing sequence converging to $0$. Define $\varphi\coloneqq \inf_{j}\varphi^j$. 
\begin{enumerate}
    \item $\int_X \theta_\varphi^n = \lim_{j\to\infty} \int_X (\theta+\epsilon_j \omega+\ddc\varphi^j)^n.$
    \item Assume that $\varphi$ has positive mass. Then for any prime divisor $E$ over $X$,
    \[
    \lim_{j\to\infty}\nu(\varphi^j,E)=\nu(\varphi,E).
    \]
\end{enumerate}

\end{lemma}

\begin{proof} 
We may assume that $X$ is smooth. 

In this case, the result is proved in \cite[Proposition~2.1]{Xia23}.
\end{proof}

\subsection{The trace operator}\label{subsec:trace}
The results in this section are simple extensions of the results established in \cite{Xia23}. The proofs are almost identical, so we omit the details.

Assume that $X$ is smooth.
Let $E$ be an irreducible reduced closed subspace of $X$ of dimension $m$. Let $\tilde{E}\rightarrow E$ denote the normalization of $E$. It is well-known that $\tilde{E}$ is a normal K\"ahler space, see \cite[Proposition~3.5]{GK20}.

Consider $\varphi\in \PSH(X,\theta)$ with generic Lelong number $\nu(\varphi,E)=0$. The trace operator gives a canonical and non-trivial way to restrict $\varphi$ to $E$, different from the naive restriction $\varphi|_E$.

The problem with the naive restriction is that even under the assumption $\nu(\varphi,E)=0$ it can happen that $\varphi|_E\equiv -\infty$. Moreover, the naive restriction is usually poorly behaved with respect to sequences of quasi-psh functions. The trace operator is a novel way of solving these issues.

Let $\varphi\in \PSH(X,\theta)$ such that $\nu(\varphi,E)=0$. We will define $\Tr_{E}(\varphi)$ as a qpsh function on the normalization $\tilde{E}$ of $E$, well-defined up to $\mathcal{I}$-equivalence.
The definition is given in \cite{Xia23}, which we recall now. 
\begin{definition}
    Suppose that $\varphi$ has analytic singularities, then we let $\Tr_{E}(\varphi)=\varphi|_{\tilde{E}}$. In general, fix a K\"ahler form $\omega$ on $X$ and take a quasi-equisingular approximation $(\varphi_j)_j$ of $\varphi$ in $\PSH(X,\theta+\omega)$. We define $\Tr_E(\varphi)$ as a $d_S$-limit of $\varphi_j|_{\tilde{E}}$.
\end{definition}
Note that $\Tr_E(\varphi)$ is not identical to $\Tr_E^{\theta}(\varphi)$ in \cite{Xia23}, but they are $\mathcal{I}$-equivalent whenever the latter has positive volume. Hence we can freely apply the results in \cite{Xia23}.

The following result is proved in \cite[Section~3]{Xia23}.
\begin{proposition}\label{prop:tracedecseq}
    Given $\varphi\in \PSH(X,\theta)$ with $\nu(\varphi,E)=0$, its trace operator $\Tr_E(\varphi)$ as a quasi-psh function on $\tilde{E}$ is well-defined up to $\mathcal{I}$-equivalence. Moreover, $\Tr_E(\varphi)$ depends only on the $\mathcal{I}$-equivalence class of $\varphi$.

    Moreover, consider a decreasing sequence of $\mathcal{I}$-good potentials $\varphi_j\in \PSH(X,\theta)$ with pointwise limit $\varphi$ satisfying $\nu(\varphi,E)=0$. Suppose that $\varphi_j\xrightarrow{d_S}\varphi$, then
    we have
    \[
    \Tr_E(\varphi_j)\xrightarrow{d_S}\Tr_E(\varphi).
    \]
\end{proposition}
Here $\varphi_j\xrightarrow{d_S}\varphi$ means the same holds after pulling-back to an arbitrary projective resolution of singularities.

In general it is desirable to have a definition of $\Tr_E(\varphi)$ when $X$ is just unibranch. 
Unfortunately, due to the lack of Demailly regularization on $X$, the author cannot define the trace operator when $E$ is contained in the singular locus of $X$.
Subsequently, in \cref{subsec:traceop}, we have to assume that our K\"ahler space is smooth.

\section{The theories of non-Archimedean metrics}
In this section, we briefly recall two different theories of non-Archimedean metrics.

\subsection{Boucksom--Jonsson theory}\label{subsec:NAformal}
In this section, we recall the basic concepts in Boucksom--Jonsson's theory as in \cite{BJ18b}.

Let $X$ be an irreducible reduced variety over $\mathbb{C}$ of dimension $n$. Let $X^{\An}$ denote the Berkovich analytification $X^{\An}$ of $X$ with respect to the trivial valuation on $\mathbb{C}$. 

The set of real valuations on $\mathbb{C}(X)$ trivial on $\mathbb{C}$ is denoted by $X^{\val}$. The center of a valuation $v$ is the scheme-theoretic point $c=c(v)$ of $X$ such that $v\geq 0$ on $\mathcal{O}_{X,c}$ and $v>0$ on the maximal ideal $\mathfrak{m}_{X,c}$ of $\mathcal{O}_{X,c}$. The center is unique if exists. It exists if $X$ is proper.

In the remaining of this section, we assume that $X$ is projective.

As a set, $X^{\An}$ is the set of semi-valuations on $X$, in other words, real-valued valuations $v$
on irreducible reduced subvarieties $Y$ in $X$ that is trivial on $\mathbb{C}$. We call $Y$ the \emph{support} of the semi-valuation $v$.
In other words, 
\[
X^{\An}=\coprod_{Y} Y^{\val}.
\]
We will write $v_{\triv}\in X^{\An}$ for the trivial valuation on $X$: $v_{\triv}(f)=0$ for any $f\in \mathbb{C}(X)^{\times}$.
See \cite{Berk93} for more details.

We will be most interested in divisorial valuations. Recall that a divisorial valuation on $X$ is a valuation of the form $c\ord_E$, where $c\in \mathbb{Q}_{>0}$ and $E$ is a prime divisor over $X$. The set of divisorial valuations on $X$ is denoted by $X^{\Div}$. This notation is compatible with that in \cref{subsec:div}.

Given any coherent ideal $\mathfrak{a}$ on $X$ and any $v\in X^{\An}$, we define 
\begin{equation}\label{eq:va}
v(\mathfrak{a}):=\min\{v(f):f\in \mathfrak{a}_{c(v)}\}\in [0,\infty],
\end{equation}
where $c(v)$ is the center of the valuation $v$ on $X$. 

Given any valuation $v$ on $X$, the Gauss extension of $v$ is a valuation $\sigma(v)$ on $X\times \mathbb{A}^1$:
\[
\sigma(v)\left(\sum_i f_i t^i\right):=\min_i (v(f_i)+i).
\]
Here $t$ is the standard coordinate on $\mathbb{A}^1=\Spec \mathbb{C}[t]$. The key property is that when $v$ is a divisorial valuation, then so it $\sigma(v)$. See \cite[Lemma~4.2]{BHJ17}.

\subsubsection*{Non-Archimedean plurisubharmonic functions.}

Let $X$ be an irreducible complex projective variety of dimension $n$ and $L$ be a holomorphic pseudoeffective $\mathbb{Q}$-line bundle on $X$. Through the GAGA morphism $X^{\An}\rightarrow X$ of ringed spaces, $L$ can be pulled-back to an analytic line bundle $L^{\An}$ on $X$. 

Following \cite[Definition~2.18]{BJ18b}, we define $\mathcal{H}^{\gf}_{\mathbb{Q}}(L)$, the set of \emph{(rational)  generically finite Fubini--Study functions}  $\phi\colon X^{\An}\rightarrow [-\infty,\infty)$, that are of the following form:
    \begin{equation}
    \phi=\frac{1}{m}\max_j \{\log|s_j|+\lambda_j\}.
    \end{equation}
Here $m\in \mathbb{Z}_{>0}$ is an integer such that $L^{\otimes m}$ is a line bundle, the $s_j$'s are a finite collection of non-vanishing sections in $H^0(X,L^{\otimes m})$, and $\lambda_j\in \mathbb{Q}$. We followed the convention of Boucksom--Jonsson by writing $\log |s_j|(v)=-v(s_j)$.

Now we come to the main definition of this paragraph:
\begin{definition}[{\cite[Definition~4.1]{BJ18b}}]\label{def:BJpshmetric}
A psh metric on $L^{\An}$ is a function $\phi:X^{\An}\rightarrow [-\infty,\infty)$ that  is not identically $-\infty$, and  is the pointwise limit of a decreasing net $(\phi_i)_{i\in I}$, where $\phi_i\in \mathcal{H}^{\gf}_{\mathbb{Q}}(L_i^{\An})$ for some $\mathbb{Q}$-line bundles $L_i$ on $X$ satisfying $c_1(L_i)\to c_1(L)$ in $\NS^1(X)_{\mathbb{R}}$.

\end{definition}

\subsection{The transcendental theory}
Let $X$ be a reduced irreducible unibranch compact K\"ahler space of dimension $n$ and $\theta$ be a closed real $(1,1)$-form on $X$. Assume that $\PSH(X,\theta)$ is non-empty.

We briefly recall the transcendental theory of non-Archimedean metrics introduced in \cite{DXZ23}.

\begin{definition}\label{def:pshnageq0}
The space $\PSH^{\NA'}(X,\theta)$ is defined as the set consisting of maps $\Gamma \colon (-\infty,\Gamma_{\max})\rightarrow \PSH(X,\theta)$ for some $\Gamma_{\max}\in \mathbb{R}$ satisfying the following conditions:
\begin{enumerate}
    \item for each $\tau\in (-\infty,\Gamma_{\max})$, $\Gamma_{\tau}$ is $\mathcal{I}$-model;
    \item the map $(-\infty,\Gamma_{\max}) \ni \tau\mapsto \Gamma_{\tau}$ is decreasing and concave.
\end{enumerate}
We write 
\[
\Gamma_{-\infty}\coloneqq \sups_{\tau\in (-\infty,\Gamma_{\max})} \Gamma_{\tau}.
\]
We define $\PSH^{\NA}(X,\theta)_{>0}$ as the subset of  $\PSH^{\NA'}(X,\theta)$ consisting of $\Gamma$
satisfying furthermore
\begin{enumerate}
    \item[(3)] $\int_X \theta_{\Gamma_{-\infty}}^n>0$.
\end{enumerate}
\end{definition}
The curves $\Gamma\in \PSH^{\NA'}(X,\theta)$ are sometimes known as \emph{test curves}. Here we choose the Greek letter $\Gamma$ simply because it resembles the reflection of $\ell$, the notation usually used to represent a geodesic ray, in view of their duality.

The object that we are eventually interested in is the space $\PSH^{\NA}(X,\theta)$ defined in \cref{def:PSHNA1}. The set $\PSH^{\NA'}(X,\theta)$ only plays an auxiliary role. In the notation $\PSH^{\NA}(X,\theta)_{>0}$, we did not put a prime in the index. Intuitively we would like to think of $\PSH^{\NA}(X,\theta)_{>0}$ as a subset of $\PSH^{\NA}(X,\theta)$ instead of a subset of $\PSH^{\NA'}(X,\theta)$. See \cref{lma:positivemasspart}.

\begin{remark}\label{rmk:extendtc}
Sometimes it is convenient to extend the domain of definition of $\Gamma\in \PSH^{\NA'}(X,\theta)$ to the whole $\mathbb{R}$ as follows: For $\tau>\Gamma_{\max}$, we simply set $\Gamma_{\tau}=-\infty$ while for $\tau=\Gamma_{\max}$, we let 
\[
\Gamma_{\Gamma_{\max}}=\inf_{\tau<\Gamma_{\max}}\Gamma_{\tau}.
\]
In particular, for each $x\in X$, the function $\mathbb{R}\ni \tau \mapsto \Gamma_{\tau}(x)$ is either constantly $-\infty$, or a proper usc concave function. In particular, it is a closed concave function and the Legendre--Fenchel duality applies.
We will frequently make use of this convention without further explanation.    
\end{remark}

\begin{lemma}\label{lma:testcurvposmass}
    Assume that $\Gamma\in \PSH^{\NA}(X,\theta)_{>0}$, then $\int_X \theta_{\Gamma_{\tau}}^n>0$ for any $\tau\in (-\infty,\Gamma_{\max})$.
\end{lemma}
\begin{proof}
    Fix $\tau\in (-\infty,\Gamma_{\max})$, we want to show that 
    \begin{equation}\label{eq:dalethtauposmass}
        \int_X \theta_{\Gamma_{\tau}}^n>0.
    \end{equation}
    We may assume that $X$ is smooth.

    By assumption, $\Gamma_{-\infty}$ has positive mass. By \cite[Theorem~2.3]{DDNL18mono}, we have
    \[
    \int_X \theta_{\Gamma_{-\infty}}^n=\lim_{\tau\to -\infty}\int_X \theta_{\Gamma_{\tau}}^n.
    \]
    In particular, for a sufficiently small $\tau_0<\tau$, we have 
    \[
    \int_X \theta_{\Gamma_{\tau_0}}^n>0.
    \]
    Now take $\tau'\in (\tau, \Gamma_{\max})$ and $t\in (0,1)$ so that
    \[
    \tau=(1-t)\tau'+t\tau_0.
    \]
    From the concavity of $\Gamma$, we find that
    \[
    \Gamma_{\tau}\geq (1-t)\Gamma_{\tau'}+t\Gamma_{\tau_0}.
    \]
    By the monotonicity theorem \cite{WN19},
    \[
    \int_X \theta_{\Gamma_{\tau}}^n\geq \int_X \theta_{(1-t)\Gamma_{\tau'}+t\Gamma_{\tau_0}}^n\geq t^n \int_X \theta_{\Gamma_{\tau_0}}^n>0
    \]
    and \eqref{eq:dalethtauposmass} follows.
\end{proof}

\begin{definition}\label{def:PSHNA1}
    Given any K\"ahler form $\omega$ on $X$, we define transition maps 
    \[
    \begin{aligned}
    & p_{\theta,\theta+\omega}\colon \PSH^{\NA'}(X,\theta)\rightarrow \PSH^{\NA'}(X,\theta+\omega),\\
    & p_{\theta,\theta+\omega}\colon \PSH^{\NA}(X,\theta)_{>0}\rightarrow \PSH^{\NA}(X,\theta+\omega)_{>0}
    \end{aligned}
    \]
    as follows: The image of $\Gamma\colon (-\infty,\Gamma_{\max})\rightarrow \PSH(X,\theta)$ is given by
    \[
    p_{\theta,\theta+\omega}(\Gamma)\colon (-\infty,\Gamma_{\max})\rightarrow \PSH(X,\theta+\omega),\quad \tau\mapsto P_{\theta+\omega}[\Gamma_{\tau}]_{\mathcal{I}}.
    \]
    Note that $\{\PSH^{\NA'}(X,\theta+\omega),p_{\theta,\theta+\omega}\}_{\omega}$ and $\{\PSH^{\NA}(X,\theta+\omega)_{>0},p_{\theta,\theta+\omega}\}_{\omega}$ form projective systems, where the set of K\"ahler forms $\omega$'s is ordered by the reverse of the usual ordering.
    We define
    \begin{equation}\label{eq:PSHNAdef}
    \PSH^{\NA}(X,\theta)\coloneqq \varprojlim_{\omega}\PSH^{\NA'}(X,\theta+\omega),
    \end{equation}
    where the projective limit is taken in the category of sets.
\end{definition}
We will  denote the components of an element $\Gamma\in \PSH^{\NA}(X,\theta)$ by $\Gamma^{\theta+\omega}\in \PSH^{\NA'}(X,\theta+\omega)$. Observe that $\Gamma^{\theta+\omega}_{\max}$ is independent of the choice of $\omega$, we will denote this common value by $\Gamma_{\max}$.

\begin{remark}
The extension in \cref{rmk:extendtc} is compatible with transition maps in the following sense: Suppose that $\Gamma\in \PSH^{\NA}(X,\theta)$, then for any K\"ahler forms $\omega$ and $\omega'$ on $X$,
\[
P_{\theta+\omega+\omega'}
\left[\Gamma^{\theta+\omega}_{\Gamma_{\max}}\right]_{\mathcal{I}}=\Gamma^{\theta+\omega+\omega'}_{\Gamma_{\max}}.
\]
\end{remark}

\begin{lemma}
    The natural maps
    \[
    \varprojlim_{\omega}\PSH^{\NA}(X,\theta+\omega)_{>0}\rightarrow \PSH^{\NA}(X,\theta)\rightarrow \varprojlim_{\omega}\PSH^{\NA}(X,\theta+\omega) 
    \]
    are both bijective.
\end{lemma}
In the sequel, we will be constantly making use of these identifications without further mentioning.
\begin{proof}
    It is clear that both maps are injective, so it suffices to verify that the map
    \[
    \varprojlim_{\omega}\PSH^{\NA}(X,\theta+\omega)_{>0}\rightarrow  \varprojlim_{\omega}\PSH^{\NA}(X,\theta+\omega) 
    \]
    is surjective. Consider an element $\Gamma\in \varprojlim_{\omega}\PSH^{\NA}(X,\theta+\omega)$ with components $\Gamma^{\theta+\omega}\in \PSH^{\NA}(X,\theta+\omega)$. For each $\omega$, $\Gamma^{\theta+\omega}$ has components $\Gamma^{\theta+\omega,\theta+\omega+\omega'}\in \PSH^{\NA'}(X,\theta+\omega+\omega')$. It is easy to verify that $\Gamma$ is the image of the following element in $\varprojlim_{\omega}\PSH^{\NA}(X,\theta+\omega)_{>0}$ with components $\Gamma^{\theta+\omega/2,\theta+\omega}$.
\end{proof}
Also note that the transition maps give natural injections 
\begin{equation}\label{eq:caninjec}
\PSH^{\NA}(X,\theta)\hookrightarrow \PSH^{\NA}(X,\theta+\omega).
\end{equation}
In particular, we can define
\[
\QPSH^{\NA}(X)\coloneqq \varinjlim_{\omega}\PSH^{\NA}(X,\theta+\omega),
\]
where $\omega$ runs over all K\"ahler forms on $X$. This definition is clearly independent of the choice of $\theta$.

\begin{lemma}\label{lma:canonicalident}
    The natural map
    \begin{equation}\label{eq:canonicalcomp1}
    \varprojlim_{\eta}\PSH^{\NA}(X,\theta+\eta)_{>0}\rightarrow \PSH^{\NA}(X,\theta)
    \end{equation}
    is a bijection, where $\eta$ runs over the set of semipositive closed real $(1,1)$-forms on $X$ with positive total mass.
\end{lemma}
\begin{proof}
    The injectivity is trivial. In order to prove the surjectivity, let $\Gamma\in \PSH^{\NA}(X,\theta)$. We want to show that $\Gamma$ can be extended to an element in the domain of \eqref{eq:canonicalcomp1}. For this purpose, take a semipositive closed real $(1,1)$-form $\eta$ on $X$ with positive total mass. Fix a K\"ahler form $\omega$ on $X$. We define
    \[
    \Gamma^{\theta+\eta}_{\tau}\coloneqq \inf_{k\in \mathbb{Z}_{>0}}\Gamma^{\theta+\eta+k^{-1}\omega}_{\tau}
    \]
    for any $\tau\in (-\infty,\Gamma_{\max})$. By \cref{lma:Imodeldeclelong}, 
    \[
    \int_X (\theta+\eta+\ddc \Gamma^{\theta+\eta}_{\tau})^n>0
    \]
    and 
    \[
    \nu\left(\Gamma^{\theta+\eta}_{\tau},E\right)=\nu\left(\Gamma^{\theta+\eta+\omega}_{\tau},E\right)
    \]
    for any prime divisor $E$ over $X$ and any $\tau\in (-\infty,\Gamma_{\max})$. It is clear that $\Gamma^{\eta}\in \PSH^{\NA}(X,\theta+\eta)_{>0}$. The $(\Gamma^{\theta+\eta})_{\eta}$ defined in this way is clearly an element in the domain of \eqref{eq:canonicalcomp1} with image $\Gamma$.
\end{proof}
\begin{definition}
    The \emph{non-pluripolar mass} of $\Gamma\in \PSH^{\NA}(X,\theta)$ is defined as the limit
    \[
    \lim_{\omega} \int_X \left(\theta+\omega+\ddc\Gamma^{\theta+\omega}_{-\infty}\right)^n,
    \]
    where $\omega$ runs over the direct set of all K\"ahler forms on $X$. In other words, if we fix a K\"ahler form $\omega$, the limit means
    \[
    \lim_{k\to\infty} \int_X \left(\theta+k^{-1}\omega+\ddc\Gamma^{\theta+k^{-1}\omega}_{-\infty}\right)^n,
    \]
\end{definition}
Since the net is decreasing, the limit exists.

\begin{lemma}\label{lma:positivemasspart}
The image of the canonical injection
\[
\PSH^{\NA}(X,\theta)_{>0}\hookrightarrow \PSH^{\NA}(X,\theta)
\]
is given by the set of  $\Gamma\in \PSH^{\NA}(X,\theta)$ with positive non-pluripolar mass.
\end{lemma}
\begin{proof}
It is clear that the image of an element in $\PSH^{\NA}(X,\theta)_{>0}$ has positive non-pluripolar mass. Conversely, take $\Gamma\in \PSH^{\NA}(X,\theta)$ with positive non-pluripolar mass. We want to construct $\Gamma'\in \PSH^{\NA}(X,\theta)_{>0}$ representing $\Gamma$.

Fix a K\"ahler form $\omega$ on $X$.
    Using the same arguments as \cref{lma:testcurvposmass}, we find that 
    \[
    \lim_{k\to\infty}\int_X \left(\theta+k^{-1}\omega+\ddc\Gamma^{\theta+k^{-1}\omega}_{\tau}\right)^n>0
    \]
    for any $\tau<\Gamma_{\max}$. We define 
    \[
    \Gamma'_{\tau}\coloneqq \lim_{k\to\infty}\Gamma^{\theta+k^{-1}\omega}_{\tau}
    \]
    for any $\tau<\Gamma_{\max}$. It follows from \cref{lma:Imodeldeclelong} that $\Gamma'$ represents $\Gamma$.
\end{proof}

\subsection{Comparison between the two theories}
Let $X$ be a reduced irreducible unibranch projective variety over $\mathbb{C}$ and $L$ be a pseudoeffective line bundle on $X$. Recall that in this case, the Zariski unibranchness is equivalent to the analytic unibranchness, as proved in \cite{Xia19}, so there is no ambiguity in the adjective \emph{unibranch}. 

Choose a smooth Hermitian metric $h_0$ on $L$ and let $\theta=c_1(L,h_0)$. In this case, both the Boucksom--Jonsson theory and the transcendental theory make sense, so we can consider the problem of comparison. 

\subsubsection{The general case}

\begin{definition}\label{def:napotentialdaleth}
    Given $\Gamma\in \PSH^{\NA}(X,\theta)$, we define its \emph{associated non\hyp{}Archimedean potential} $\Gamma^{\An}\colon X^{\Div}\rightarrow \mathbb{R}$ as follows:
    \begin{equation}\label{eq:dalethandiv}
    \Gamma^{\An}(c\ord_E) \coloneqq \sup_{\tau<\Gamma_{\max}} \left(\Gamma^{\theta+\omega,\An}_{\tau}(c\ord_E)+\tau\right)
    \end{equation}
    for any K\"ahler form $\omega$ on $X$. Clearly, this map is independent of the choice of $\omega$. 
\end{definition}
The map $\Gamma\mapsto \Gamma^{\An}$ is compatible with the inclusions \eqref{eq:caninjec}. In particular, given $\Gamma\in \QPSH^{\NA}(X)$, we can define $\Gamma^{\An}$ by choosing any representative of $\Gamma$. 

It is clear that the map $\Gamma\mapsto \Gamma^{\An}$ is injective.

\subsubsection{On a smooth variety}

In the sequel we assume that $X$ is smooth.  In this case, we have the techniques of multiplier ideal sheaves.

A potential $\varphi\in \PSH(X,\theta)$ defines a potential $\varphi^{\An}\in \PSH(L^{\An})$ as follows:
\[
\varphi^{\An}(v)\coloneqq -\lim_{k\to\infty} \frac{1}{k}v\left(\mathcal{I}(k\varphi)\right)
\]
for any $v\in X^{\An}$. As explained in \cite{BBJ21, DXZ23}, the limit exists and lies in $\PSH(L^{\An})$ and this function extends the non-Archimedean potential defined in \cref{def:napotentialdiv} on $X^{\Div}$. 

We can then define a map
\begin{equation}\label{eq:comparison}
\An\colon \PSH^{\NA}(X,\theta)\rightarrow \PSH(L^{\An})
\end{equation}
as follows: Take $\Gamma\in \PSH^{\NA}(X,\theta)$, its image is given by
\[
\Gamma^{\An}\colon X^{\An}\rightarrow [-\infty,\infty),\quad v\mapsto \sup_{\tau<\Gamma_{\max}} \left(\Gamma^{\theta+\omega,\An}_{\tau}(v)+\tau\right)
\]
for any K\"ahler form $\omega$ on $X$. Clearly, this map is independent of the choice of $\omega$. The map $\Gamma^{\An}$ extends the corresponding definition in \cref{def:napotentialdaleth}.

\begin{theorem}[{\cite{DXZ23}}]\label{thm:DXZ}
   Assume that $X$ is smooth, then the map \eqref{eq:comparison} is a bijection. 
\end{theorem}

Motivated by this bijection, we could translate the pluripotential-theoretic operations of $\PSH(L^{\An})$ to operations on the space $\PSH^{\NA}(X,\theta)$. Furthermore, in most cases, the resulting operations have natural extensions to the transcendental case as well. This is what we will carry out in the next section.

\section{Pluripotential-theoretic operations}
Let $X$ be a reduced irreducible unibranch compact K\"ahler space of dimension $n$ and $\theta$ be a closed real $(1,1)$-form on $X$. Assume that $\PSH(X,\theta)\neq \emptyset$.  When necessary, we use $\theta'$ and $\theta''$ to denote similar forms with the same properties as $\theta$.

\subsection{The partial order}
\begin{definition}
    We define a partial order $\leq$ on $\PSH^{\NA}(X,\theta)$ as follows: Given $\Gamma,\Gamma'\in \PSH^{\NA}(X,\theta)$, we say $\Gamma\leq \Gamma'$ if for some (hence for every) K\"ahler form $\omega$ on $X$, we have $\Gamma^{\theta+\omega}_{\tau}\leq \Gamma'^{\theta+\omega}_{\tau}$ for all $\tau<\Gamma_{\max}$.
\end{definition}
It is trivial to verify that this is indeed a partial order. Note that this partial order is compatible with the inclusions \eqref{eq:caninjec}, so it defines a partial order on $\QPSH^{\NA}(X)$.

\begin{lemma}\label{lma:valuationconc}
Suppose that $\Gamma,\Gamma'\in \PSH^{\NA}(X,\theta)$, then the following are equivalent:
\begin{enumerate}
    \item $\Gamma\leq \Gamma'$;
    \item $\Gamma^{\An}(c\ord_E)\leq \Gamma'^{\An}(c\ord_E)$ for all $c\ord_E\in X^{\Div}$.
\end{enumerate}
\end{lemma}
The proof is almost identical to \cite[Theorem~3.12]{DXZ23}, we omit the details.

\subsection{Addition by a constant}
\begin{definition}
    Given $\Gamma\in \PSH^{\NA}(X,\theta)$ and $C\in \mathbb{R}$, we define the addition $\Gamma+C\in \PSH^{\NA}(X,\theta)$ as the element with components $(\Gamma+C)^{\theta+\omega}\colon (-\infty,\Gamma_{\max}+C)\rightarrow \PSH(X,\theta+\omega)$ given by
    \[
    (\Gamma+C)^{\theta+\omega}_{\tau}\coloneqq \Gamma^{\theta+\omega}_{\tau-C}.
    \]
\end{definition}
The associated non-Archimedean potential behaves as expected:
\begin{proposition}
    Given $\Gamma\in \PSH^{\NA}(X,\theta)$ and $C\in \mathbb{R}$, we have
    \begin{equation}
        (\Gamma+C)^{\An}(c\ord_E)=\Gamma^{\An}(c\ord_E)+C
    \end{equation}
    for all $c\ord_E\in X^{\Div}$ and 
    \[
    (\Gamma+C)_{\max}=\Gamma_{\max}+C.
    \]
\end{proposition}
This follows trivially from \eqref{eq:dalethandiv}. We will omit similar proofs in the sequel.

Of course, this operation has the obvious compatibility in itself: Given another constant $C'\in \mathbb{R}$, we have
\[
(\Gamma+C)+C'=\Gamma+(C+C').
\]
It is compatible with the partial order in the following sense: Given another $\Gamma'\in \PSH^{\NA}(X,\theta)$, the following are equivalent:
\begin{enumerate}
    \item $\Gamma\leq \Gamma'$;
    \item $\Gamma+C\leq \Gamma'+C$.
\end{enumerate}

\subsection{Maximum}

\begin{definition}
    Suppose that $\Gamma,\Gamma'\in \PSH^{\NA}(X,\theta)_{>0}$, we define their \emph{maximum} $\Gamma\lor \Gamma'\in \PSH^{\NA}(X,\theta)_{>0}$ as the map $(-\infty,\Gamma_{\max}\lor \Gamma'_{\max})\rightarrow \PSH(X,\theta)$:
    \[
    \tau\mapsto P_{\theta}[\EC(\tau'\mapsto \Gamma_{\tau'}\lor \Gamma'_{\tau'})_{\tau}]_{\mathcal{I}},
    \]
    where $\EC$ denotes the concave hull.
\end{definition}
We need to verify that $\Gamma\lor \Gamma'\in \PSH^{\NA}(X,\theta)_{>0}$. The only non-trivial part is to show that the concave envelope $\EC(\tau'\mapsto \Gamma_{\tau'}\lor \Gamma'_{\tau'})$ is a curve of $\theta$-psh functions. For this purpose, we may assume that $X$ is smooth and connected. According to the Ross--Witt Nystr\"om correspondence (\cite[Section~5, Section~6]{RWN14} and \cite[Theorem~2.6]{DXZ23}), $\Gamma$ and $\Gamma'$ correspond to geodesic rays $\ell$ and $\ell'$ in $\PSH(X,\theta)$: For any $t\geq 0$, we have
\[
\ell_t=\sup_{\tau\in \mathbb{R}} (\Gamma_{\tau}+t\tau),\quad \ell'_t=\sup_{\tau\in \mathbb{R}} (\Gamma'_{\tau}+t\tau).
\]
It follows that
\[
\ell_t\lor \ell'_t=\sup_{\tau\in \mathbb{R}} (\Gamma_{\tau}\lor \Gamma'_{\tau}+t\tau).
\]
Hence,
\[
\EC(\tau\mapsto \Gamma_{\tau}\lor \Gamma'_{\tau})_{\tau'}=\inf_{t\geq 0}(\ell_t\lor \ell'_t-t\tau')
\]
and $\EC(\tau\mapsto \Gamma_{\tau}\lor \Gamma'_{\tau})_{\tau'}\in \PSH(X,\theta)\cup \{-\infty\}$ by Kiselman's principle.

By a similar argument, we have
\begin{corollary}\label{cor:napotmax}
Suppose that $\Gamma,\Gamma'\in \PSH^{\NA}(X,\theta)_{>0}$. For any $c\ord_E\in X^{\Div}$, we have
\[
-(\Gamma\lor\Gamma')_{\tau}^{\An}(c\ord_E)=\EC\left(\tau'\mapsto \left(-\Gamma_{\tau'}^{\An}(c\ord_E)\right)\lor \left(-\Gamma'^{\An}_{\tau'}(c\ord_E)\right) \right)_{\tau}.
\]
In particular,
\begin{equation}\label{eq:loranmax}
(\Gamma\lor\Gamma')^{\An}(c\ord_E)=\Gamma^{\An}(c\ord_E)\lor \Gamma'^{\An}(c\ord_E).
\end{equation}
\end{corollary}

It follows from \cref{cor:napotmax} and \cref{lma:valuationconc} that the maximum operation is compatible with the projective system \eqref{eq:PSHNAdef}. Hence, we can make the following definition in general:
\begin{definition}
    Suppose that $\Gamma,\Gamma'\in \PSH^{\NA}(X,\theta)$, we define their \emph{maximum} $\Gamma\lor \Gamma'\in \PSH^{\NA}(X,\theta)$ as the element with components 
    \[
    (\Gamma\lor \Gamma')^{\theta+\omega}=\Gamma^{\theta+\omega}\lor \Gamma'^{\theta+\omega}.
    \]
\end{definition}
The associated non-Archimedean potential behaves as expected:
\begin{corollary}\label{cor:napotmaxgen}
Suppose that $\Gamma,\Gamma'\in \PSH^{\NA}(\theta)$. For any $c\ord_E\in X^{\Div}$, we have
\begin{equation}\label{eq:loranmaxgen}
(\Gamma\lor\Gamma')^{\An}(c\ord_E)=\Gamma^{\An}(c\ord_E)\lor \Gamma'^{\An}(c\ord_E)
\end{equation}
and 
\[
(\Gamma\lor\Gamma')_{\max}=\Gamma_{\max}\lor \Gamma'_{\max}.
\]
\end{corollary}
The maximum operation is commutative and associative. In particular, finite (non-empty) maximum makes sense in the obvious way. 

The maximum operation is compatible with the partial order in the following sense: Given $\Gamma,\Gamma'\in \PSH^{\NA}(\theta)$, the following are equivalent
\begin{enumerate}
    \item $\Gamma\lor \Gamma'=\Gamma'$;
    \item $\Gamma\leq \Gamma'$.
\end{enumerate}

The maximum operation is compatible with the addition by a constant: If $C\in \mathbb{R}$, then
\[
(\Gamma\lor \Gamma')+C=(\Gamma+C)\lor (\Gamma'+C).
\]

\subsection{Decreasing limit along a net}
\begin{definition}
    Let $(\Gamma_i)_{i\in I}$ be a decreasing net in $\PSH^{\NA}(X,\theta)$. Assume that $\inf_{i\in I}\Gamma_{i,\max}>-\infty$. We define $\inf_{i\in I}\Gamma_i\in \PSH^{\NA}(X,\theta)$ as the element with components 
    \[
    \left(\inf_{i\in I}\Gamma_i \right)^{\theta+\omega}_{\tau}=\inf_{i\in I} \Gamma^{\theta+\omega}_{i,\tau}.
    \]
    This gives an element in $\PSH^{\NA}(X,\theta)$ by \cref{prop:decnetPproj} together with \cite[Lemma~2.21(i)]{DX22}\footnote{The result is stated for a sequence, but the proof extends to a net.}. 
\end{definition}
When $\inf_{i\in I}\Gamma_{i,\max}=-\infty$, we could still formally define $\inf_{i\in I}\Gamma_i$ as a symbol $-\infty$. As one could easily verify, all of our pluripotential-theoretic operations admit natural extensions to $-\infty$.

Using \cite[Lemma~3.13]{DXZ23}, we have
\begin{lemma}\label{lma:decnetna}
    Let $(\Gamma_i)_{i\in I}$ be a decreasing net in $\PSH^{\NA}(X,\theta)$ with $\inf_{i\in I}\Gamma_i\in \PSH^{\NA}(X,\theta)$. Then for any $c\ord_E\in X^{\Div}$,
    \begin{equation}
        \left(\inf_{i\in I}\Gamma_i\right)^{\An}(c\ord_E)=\inf_{i\in I}\Gamma_i^{\An}(c\ord_E).
    \end{equation}
    Moreover,
    \[
    \left(\inf_{i\in I}\Gamma_i \right)_{\max}=\inf_{i\in I}\Gamma_{i,\max}.
    \]
\end{lemma}

This operation is compatible with itself: If $I$ and $J$ are two non-empty directed set and $\Gamma_{ij}$ is a decreasing net in $\PSH^{\NA}(X,\theta)$ indexed by $I\times J$ with $\inf_{i,j}\Gamma_{ij}\in \PSH^{\NA}(X,\theta)$, then
\[
\inf_{i,j}\Gamma_{ij}=\inf_{i\in I}\inf_{j\in J}\Gamma_{ij}=\inf_{j\in J}\inf_{i\in I}\Gamma_{ij}.
\]

This operation is compatible with the partial order in the following senses: If $\Gamma_i\in \PSH^{\NA}(X,\theta)$ ($i\in I$) is a decreasing net with $\inf_{i\in I}\Gamma_i\in \PSH^{\NA}(X,\theta)$. Then $\inf_{j\in I}\Gamma_j\leq \Gamma_i$ for all $i\in I$ and it is the biggest element in $\PSH^{\NA}(X,\theta)$ with this property. On the other hand, if $\Gamma'_i\in \PSH^{\NA}(X,\theta)$ ($i\in I$) is another decreasing net with $\inf_{i\in I}\Gamma'_i\in \PSH^{\NA}(X,\theta)$ such that $\Gamma_i\leq \Gamma'_i$ for all $i\in I$, then
\[
\inf_{i\in I}\Gamma_i\leq \inf_{i\in I}\Gamma'_i.
\]

This operation is compatible with the addition by a constant: If $C\in \mathbb{R}$, then
\[
\inf_{i\in I}(\Gamma_i+C)=\left(\inf_{i\in I}\Gamma_i\right)+C.
\]

\subsection{Regularized supremum}
\begin{definition}
    Let $(\Gamma^i)_{i\in I}$ be an increasing net in $\PSH^{\NA}(X,\theta)_{>0}$. We assume that $\sup_{i\in I}\Gamma^i_{\max}<\infty$. We define a map $\sup_{i\in I}\Gamma^i_{\max}\colon (-\infty,\sup_i\Gamma^i_{\max})\rightarrow \PSH(X,\theta)$ as follows:
\[
\left(\sups_{i\in I} \Gamma^i\right)_{\tau}=P_{\theta}\left[\sups_{i\in I} \Gamma^i_{\tau}\right].
\]
By \cref{lma:supIbon}, it is easy to see that $\sups_i \Gamma^i\in \PSH^{\NA}(X,\theta)_{>0}$.

More generally, suppose that $\Gamma^i\in \PSH^{\NA}(X,\theta)$ ($i\in I$) is an increasing net with $\sup_i\Gamma^i_{\max}<\infty$.
We define $\sups_i \Gamma^i\in \PSH^{\NA}(X,\theta)$ as the element with components 
\begin{equation}\label{eq:Psupext}
\left(\sups_{i\in I} \Gamma^i\right)^{\theta+\omega}=\sups_{i\in I} \Gamma^{i,\theta+\omega}
\end{equation}
for any K\"ahler form $\omega$ on $X$. The compatibility follows from \cite[Lemma~3.14]{DX21}.
\end{definition}

\begin{definition}
    Let $(\Gamma^i)_{i\in I}$ be a (non-empty) family in $\PSH^{\NA}(X,\theta)$ such that $\sup_{i\in I}\Gamma^i_{\max}<\infty$. Let $\mathcal{J}$ denote the directed set of finite non-empty subsets of $I$ ordered by inclusion. Then we define
    \[
    \sups_{i\in I}\Gamma^i\coloneqq \sups_{J\in \mathcal{J}} \left(\max_{j\in J}\Gamma^j\right).
    \]
\end{definition}
Recall that the maximum is defined right after \cref{cor:napotmaxgen}.

\begin{theorem}\label{thm:enve}
    Let $(\Gamma^i)_{i\in I}$ be a (non-empty) family in $\PSH^{\NA}(X,\theta)$ with $\sup_{i\in I}\Gamma^i_{\max}<\infty$. Then for any $c\ord_E\in X^{\Div}$, we have
    \[
    \left(\sups_{i\in I}\Gamma^i\right)^{\An}(c\ord_E)=\sup_{i\in I}\Gamma^{i,\An}(c\ord_E).
    \]
    Moreover,
    \[
    \left(\sups_{i\in I}\Gamma^i\right)_{\max}=\sup_{i\in I}\Gamma^i_{\max}.
    \]
\end{theorem}
This result is the transcendental analogue of Boucksom--Jonsson's envelope conjecture. As explained in \cite[Theorem~3.4]{DXZ23}, when $X$ is smooth and projective, this result implies Boucksom--Jonsson's envelope conjecture.
\begin{proof}
    By \cref{cor:napotmaxgen}, we may assume that $(\Gamma^i)_{i\in I}$ is an increasing net. Furthermore, it is easy to reduce to the case where $\Gamma^i\in \PSH^{\NA}(X,\theta)_{>0}$. In fact, by definition,
    \[
    \left(\sups_{i\in I}\Gamma^i\right)^{\An}=\left(\sups_{i\in I}\Gamma^{i,\theta+\omega}\right)^{\An}
    \]
    for any K\"ahler form $\omega$ on $X$. A similar equation holds for $(\bullet)_{\max}$ as well. So we could replace $\Gamma^i$ by $\Gamma^{i,\theta+\omega}$ and assume that $\Gamma^i\in \PSH^{\NA}(X,\theta)_{>0}$.

    In this case, the result follows from \cite[Lemma~3.13]{DXZ23}.
\end{proof}
In particular, this operation extends the finite maximum.

This operation is compatible with itself: If $(\Gamma^{i,j})_{i\in I,j\in J}$ is a (non-empty) family in $\PSH^{\NA}(X,\theta)$ with $\sup_{i\in I,j\in J} \Gamma^{i,j}_{\max}<\infty$. Then
\[
\sups_{i\in I}\sups_{j\in J}\Gamma^{i,j}=\sups_{j\in J}\sups_{i\in I}\Gamma^{i,j}.
\]
It is compatible with the partial order: Suppose that $(\Gamma^i)_{i\in I},(\Gamma'^i)_{i\in I}$ are two families in $\PSH^{\NA}(X,\theta)$ with $\sups_{i\in I}\Gamma^i_{\max}<\infty$ and $\sups_{i\in I}\Gamma'^i_{\max}<\infty$ and $\Gamma^i\leq \Gamma'^i$ for each $i\in I$, then 
\[
\sups_{i\in I}\Gamma^i\leq \sups_{i\in I}\Gamma'^i.
\]
It is compatible with addition by a constant: If $C\in \mathbb{R}$, then
\[
\sups_{i\in I}(\Gamma^i+C)=\left(\sups_{i\in I}\Gamma^i\right)+C.
\]
There is no non-trivial compatibility with the decreasing limits, as expected.

\subsection{Rescaling and addition}
\begin{definition}
    Let $\Gamma\in \PSH^{\NA}(X,\theta)$ and $\lambda\in \mathbb{R}_{>0}$, we define the rescaling $\lambda\Gamma\in \PSH^{\NA}(X,\lambda\theta)$ as the element with components $(\lambda\Gamma)^{\lambda\theta+\omega}\colon (-\infty,\lambda \Gamma_{\max})\rightarrow \PSH(X,\lambda\theta+\omega)$:
    \[
    (\lambda\Gamma)^{\lambda\theta+\omega}_{\tau}\coloneqq \lambda \Gamma^{\theta+\lambda^{-1}\omega}_{\lambda^{-1}\tau}
    \]
for any K\"ahler form $\omega$ on $X$.
\end{definition}

\begin{definition}
    Let $\Gamma\in \PSH^{\NA}(X,\theta)_{>0}$ and $\Gamma'\in \PSH^{\NA}(X,\theta')_{>0}$. We define $\Gamma+\Gamma'\colon (-\infty,\Gamma_{\max}+ \Gamma'_{\max})\rightarrow \PSH(X,\theta+\theta')$ as 
    \[
    (\Gamma+\Gamma')_{\tau}\coloneqq P_{\theta+\theta'}\left[\sup_{t\in \mathbb{R}} \left(\Gamma_t+\Gamma'_{\tau-t}\right)\right].
    \]
\end{definition}
Observe that $\tau\mapsto \sup_{t\in \mathbb{R}} (\Gamma_t+\Gamma'_{\tau-t})$ is concave. In fact, this curve is just the concave involution of $\Gamma$ and $\Gamma'$. From the general properties of infimal involution, we know that this map is concave and usc. For each fixed $x\in X$, $\tau\mapsto \sup_{t\in \mathbb{R}} (\Gamma_t(x)+\Gamma'_{\tau-t}(x))$ is either constantly $-\infty$ or proper. In particular, by \cite[Theorem~16.4]{Roc70} and Legendre--Fenchel duality, $\tau\mapsto \sup_{t\in \mathbb{R}} (\Gamma_t+\Gamma'_{\tau-t})$ is the same as the inverse Legendre transform of $\ell+\ell'$, where $\ell$ and $\ell'$ are the geodesic rays defined by $\Gamma$ and $\Gamma'$. In particular, $\sup_{t\in \mathbb{R}} (\Gamma_t+\Gamma'_{\tau-t})$ is either in $\PSH(X,\theta)$ or constantly $-\infty$ by Kiselman's principle. By \cref{lma:supIbon} and \cite[Proposition~2.8]{Xia22}, we find that $\Gamma+\Gamma'\in \PSH^{\NA}(X,\theta)_{>0}$.

This argument also shows that for any $c\ord_E\in X^{\Div}$, we have
\begin{equation}\label{eq:plusan1}
(\Gamma+\Gamma')^{\An}(c\ord_E)=\Gamma^{\An}(c\ord_E)+\Gamma'^{\An}(c\ord_E).
\end{equation}

\begin{definition}
    Let $\Gamma\in \PSH^{\NA}(X,\theta)$ and $\Gamma'\in \PSH^{\NA}(X,\theta')$. We define $\Gamma+\Gamma'\in  \PSH^{\NA}(X,\theta+\theta')$ as the element with components 
    \[
    (\Gamma+\Gamma')^{\theta+\omega}=\Gamma^{\theta+\omega}+\Gamma'^{\theta+\omega}.
    \]
\end{definition}
Note that $\Gamma+\Gamma'\in \PSH^{\NA}(X,\theta+\theta')$ by \eqref{eq:plusan1}.

\begin{lemma}
    Suppose that $\Gamma\in \PSH^{\NA}(X,\theta)$, $\Gamma'\in \PSH^{\NA}(X,\theta')$, $\lambda\in \mathbb{R}_{>0}$. Then for any $c\ord_E\in X^{\Div}$, we have
    \[
    (\lambda\Gamma)^{\An}(c\ord_E)=\lambda \Gamma^{\An}(c\ord_E),\quad (\Gamma+\Gamma')^{\An}(c\ord_E)=\Gamma^{\An}(c\ord_E)+\Gamma'^{\An}(c\ord_E).
    \]
    Moreover, 
    \[
    (\lambda \Gamma)_{\max}=\lambda \Gamma_{\max},\quad (\Gamma+\Gamma')_{\max}=\Gamma_{\max}+\Gamma'_{\max}.
    \]
\end{lemma}
Both operations are compatible with themselves: If $\lambda,\lambda'\in \mathbb{R}_{>0}$ and $\Gamma\in \PSH^{\NA}(X,\theta)$, $\Gamma'\in \PSH^{\NA}(X,\theta')$ and $\Gamma''\in \PSH^{\NA}(X,\theta'')$, then
\[
\lambda(\lambda'\Gamma)=(\lambda\lambda')\Gamma
\]
and
\[
(\Gamma+\Gamma')+\Gamma''=\Gamma+(\Gamma'+\Gamma'').
\]
The addition is clear commutative. Moreover, these two operations are compatible with each other:
\[
\lambda(\Gamma+\Gamma')=\lambda\Gamma+\lambda\Gamma',\quad (\lambda+\lambda')\Gamma=\lambda\Gamma+\lambda'\Gamma.
\]
Both operations are compatible with the partial order: If $\Gamma,\Gamma'\in \PSH^{\NA}(X,\theta)$, $\lambda\in \mathbb{R}_{>0}$, then $\Gamma\leq \Gamma'$ if and only if $\lambda\Gamma\leq \lambda\Gamma'$. Suppose that $\Gamma^1,\Gamma^2\in \PSH^{\NA}(X,\theta)$, $\Gamma'^1,\Gamma'^2\in \PSH^{\NA}(X,\theta')$ and $\Gamma^1\leq \Gamma^2$, $\Gamma'^1\leq \Gamma'^2$, then $\Gamma^1+\Gamma'^1\leq \Gamma^2+\Gamma'^2$.

Both operations are compatible with addition by a constant: If $C\in \mathbb{R}$, $\lambda\in \mathbb{R}_{>0}$, $\Gamma\in \PSH^{\NA}(X,\theta)$ and $\Gamma'\in \PSH^{\NA}(X,\theta')$, then
\[
\lambda(\Gamma+C)=\lambda \Gamma+\lambda C
\]
and
\[
(\Gamma+\Gamma')+C=\Gamma+(\Gamma'+C).
\]
Both operations satisfy obvious inequalities with respect to the limit operations. We leave the details to the readers.

\subsection{The trace operator}\label{subsec:traceop}
We assume in addition that $X$ is smooth. Recall that in this case, we have a natural bijection \eqref{eq:comparison}. Fix an irreducible reduced closed analytic subspace $E$ of $X$. 

\begin{definition}
Assume that $X$ is smooth.
    We say $\Gamma\in \PSH^{\NA}(X,\theta)$ has \emph{non-trivial restriction} to $E$ if there is a small enough $C\in \mathbb{R}$ so that $\nu(\Gamma^{\theta+\omega}_{C},E)=0$ for any K\"ahler form $\omega$ on $X$. 
\end{definition}

This definition is motivated by the following observation.
\begin{lemma}\label{lma:non-trivrest}
    Assume that $X$ is smooth and projective.
    Let $\Gamma\in \PSH^{\NA}(X,\theta)$. Then the following are equivalent:
    \begin{enumerate}
    \item $\Gamma^{\An}(v_{E,\triv})\neq -\infty$;
    \item $\Gamma^{\An}|_{E^{\An}}\not\equiv -\infty$;
    \item $\Gamma$ has non-trivial restriction to $E$.
\end{enumerate}
\end{lemma}
Here $v_{E,\triv}$ denotes the trivial valuation of $\mathbb{C}(E)$.

\begin{proof}
The equivalence between (1) and (2) is a simple consequence of the maximum principle \cite[Lemma~1.4(i)]{BJ18b}. 

To see the equivalence between (1) and (3), it suffices to observe that for any $\varphi\in \PSH(X,\theta)$,
\[
\varphi^{\An}(v_{E,\triv})=
\left\{
\begin{aligned}
    -\infty, & \text{ if }\nu(\varphi,E)>0;\\
    0,& \text{ if }\nu(\varphi,E)=0.
\end{aligned}
\right.
\]
\end{proof}

\begin{definition}
Assume that $X$ is smooth.
    Assume that $\Gamma\in \PSH^{\NA}(X,\theta)$ has non-trivial restriction to $E$, we define the \emph{trace operator} of $\Gamma$ along $E$ as the element $\Tr_E(\Gamma)\in \PSH^{\NA}(\tilde{E},\theta|_{\tilde{E}})$ defined as follows: For any K\"ahler form $\omega'$ on $Y$ and any K\"ahler form $\omega$ on $X$ satisfying $\omega'\geq \omega|_{\tilde{Y}}$, the component $\Tr_E(\Gamma)^{\theta|_{\tilde{E}}+\omega'}\in \PSH^{\NA'}(\tilde{E},\theta|_{\tilde{E}}+\omega')$ is given by
    \[
     \tau\mapsto P_{\theta|_{\tilde{E}}+\omega'}[\Tr_E(\Gamma^{\theta+\omega}_{\tau})],
    \]
    where $\tau\in (-\infty,(\Tr_E(\Gamma))_{\max})$ and $(\Tr_E(\Gamma))_{\max}=\sup \{\tau\in \mathbb{R}:\nu(\Gamma^{\theta+\omega''}_{\tau},E)=0\}$ for any K\"ahler form $\omega''$ on $X$.
\end{definition}
We remind the readers that $\tilde{E}$ is the normalization of $E$. It is a normal K\"ahler space by \cite[Proposition~3.5]{GK20}.

The non-Archimedean potential associated with $\Tr_E(\Gamma)$ is not very easy to describe. We will only do this in the following special but important case.
\begin{theorem}\label{thm:comp_traceop} Assume that $X$ is smooth projective and $\{\theta\}$ is the first Chern class of a pseudoeffective line bundle $L$.
    Consider an element $\Gamma\in \PSH^{\NA}(X,\theta)$ with non-trivial restriction to $E$. Then 
    \begin{equation}\label{eq:Tracena}
    \Tr_E(\Gamma)^{\An}|_{E^{\Div}}=\Gamma^{\An}|_{E^{\Div}}.
    \end{equation}
\end{theorem}
Observe that there is a canonical identification $E^{\Div}=\tilde{E}^{\Div}$. 

In order to prove the theorem, we need to recall the following notion introduced in \cite{DXZ23}. 
\begin{definition}
    A \emph{piecewise linear curve} $\psi_{\bullet}$ in $\PSH(X,\theta)$ associated with $\psi_{\tau_j} \in \PSH(X,\theta)$, for a finite number of parameters $\tau_0>\tau_1> \dots > \tau_N$ is the affine interpolation of these data:
    \begin{enumerate}
        \item $\psi_{\tau}=\psi_{\tau_N}$ for $\tau\leq \tau_N$;
        \item For $t\in (0,1)$ and $i=0,\ldots,N-1$, we have $\psi_{(1-t)\tau_i+t\tau_{i+1}}=(1-t)\psi_{\tau_i}+t\psi_{\tau_{i+1}}$;
        \item $\psi_{\tau}=-\infty$ for $\tau>\tau_0$.
    \end{enumerate}
\end{definition}
The analytification $\psi^\An$ of a piecewise linear curve $\psi_{\bullet}$ as above is defined as
\begin{equation}\label{eq:psianvlinear}
\psi^{\An}(v)\coloneqq \sup_{\tau\leq \tau_0} (\psi_{\tau}^{\An}(v)+\tau)=\max_{\tau_i} (\psi_{\tau_i}^{\An}(v)+\tau_i)\quad \text{for }v\in X^{\An}.
\end{equation}

\begin{proof}[Proof of \cref{thm:comp_traceop}]
    We may assume that $\Gamma\in \PSH^{\NA}(X,\theta)_{>0}$.
    Let $\phi=\Gamma^{\An}\in \PSH(L^{\An})$. By \cref{lma:non-trivrest}, $\phi(v_{E,\triv})\neq -\infty$. 
    
    Take an ample line bundle $S$ on $X$, a K\"ahler form $\omega$ in $c_1(S)$.
    Write $\phi$ as the decreasing limit of a sequence $\phi^{i}$ of elements in $\mathcal{H}^{\gf}_{\mathbb{Q}}(L+i^{-1}S)$. This is possible by \cite[Corollary~12.18]{BJ18b}.
    
    Take $\Gamma^{i}\in \PSH^{\NA}(X,\theta+i^{-1}\omega)$ such that $\Gamma^{i,\An}=\phi^{i}$. Note that by \cref{lma:positivemasspart}, $\Gamma^i\in \PSH^{\NA}(X,\theta+i^{-1}\omega)_{>0}$.

    It follows from \cref{lma:decnetna} (applied to the images of $\Gamma^i$ in $\PSH^{\NA}(X,\theta+\omega)$) that for any $\tau<\Gamma_{\max}$, we have
    \[
    \lim_{i\to\infty}\Gamma^{i}_{\tau}=\Gamma_{\tau}.
    \]
    In particular, by \cref{lma:Imodeldeclelong}, the Lelong numbers of $\Gamma^{i}_{\tau}$ converge to those of $\Gamma_{\tau}$. It follows that $\Gamma^{i}_{\tau}\xrightarrow{d_{S,\theta+\omega}} \Gamma_{\tau}$ for all $\tau<\Gamma_{\max}$.
    
    By \cref{lma:non-trivrest} again, each $\Gamma^{i}$ has non-trivial restriction to $E$. By \cite[Lemma~3.12]{Xia23}, for any K\"ahler form $\omega'$ on $\tilde{E}$ satisfying $\omega'\geq \omega|_{\tilde{E}}$ we have
    \[
    \Tr_E(\Gamma^{i})^{\theta|_{\tilde{E}}+\omega'}_{\tau}\xrightarrow{d_S} \Tr_E(\Gamma)^{\theta|_{\tilde{E}}+\omega'}_{\tau}
    \]
    for any $\tau<(\Tr_E(\Gamma))_{\max}$. Thanks to \cite[Theorem~0.10]{Xia22}, the Lelong numbers of $\Tr_E(\Gamma^{i})^{\theta|_{\tilde{E}}+\omega'}_{\tau}$ converge to those of $\Tr_E(\Gamma)^{\theta|_{\tilde{E}}+\omega'}_{\tau}$. Therefore, by \cite[Lemma~3.13]{DXZ23},
    \[
    \Tr_E(\Gamma)^{\An}(c\ord_F)=\inf_{i}\Tr_E(\Gamma)^{i,\An}(c\ord_F)
    \]
    for any $c\ord_F\in E^{\Div}$.
    In particular, it suffices to prove \eqref{eq:Tracena} with $\Gamma^i$ in place of $\Gamma$.
    
    In other words, we have reduced to the case where $\phi\in \mathcal{H}^{\gf}_{\mathbb{Q}}(L)$ and $L$ is big.

    Under this extra assumption, as proved in the proof of \cite[Theorem~3.12]{DXZ23}, there is a piecewise linear curve $\psi$ connecting potentials lying in the space of potentials generated by functions of the form
    \begin{equation}\label{eq:breakingpoints}
    m^{-1}\log \max_i(|s_i|_{h_0^m}^2+c)
    \end{equation}
    using convex combinations and finite maxima, 
    where $m\in \mathbb{Z}_{>0}$, the $s_i$'s are finitely many holomorphic sections of $L^{\otimes m}$ and $c\in \mathbb{Q}$ such that $\Gamma_{\tau}$ is the $\mathcal{I}$-envelope of $\eta_{\tau}$ for any $\tau<\Gamma_{\max}$, where $\eta$ is the concave envelope of $\psi$. Moreover, after adjusting $\psi$ (c.f. \cite[Theorem~5.6]{Roc70}), we may guarantee that $\psi$ itself is concave. In particular, $\Gamma_{\tau}\sim_{\mathcal{I}}\psi_{\tau}$ for all $\tau< \Gamma_{\max}$.

    It follows that for any $c\ord_F\in E^{\Div}$,
    \[
    \begin{aligned}
    \phi|_{E^{\An}}(c\ord_F)=&\sup_{\tau<\Gamma_{\max}} \left(\psi_{\tau}^{\An}(c\ord_F)+\tau\right)\\
    =& \sup_{\tau<\Gamma_{\max}} \left(\left(\psi_{\tau}|_{\tilde{E}}\right)^{\An}(c\ord_F)+\tau\right)\\
    =& \sup_{\tau<\Gamma_{\max}} \left(\Tr_E(\psi_{\tau})^{\An}(c\ord_F)+\tau\right)\\
     =& \sup_{\tau<\Gamma_{\max}} \left(\Tr_E(\Gamma_{\tau})^{\An}(c\ord_F)+\tau\right)\\
    =& \Tr_E(\Gamma)^{\An}(c\ord_F).
    \end{aligned}
    \]
    Here we adopt the convention that $(-\infty)^{\An}(c\ord_F)=-\infty$ to simplify the notations.
    The third equality follows from the linearity of the trace operator \cite[Lemma~3.10]{Xia23} and the obvious fact that the trace operator preserves finite maxima. The fourth equality follows from \cite[Lemma~3.9]{Xia23}. 
    In order to justify the second equality, it suffices to show that for $\tau<(\Tr_E(\Gamma))_{\max}$, if $\psi_{\tau}$ has the form \eqref{eq:breakingpoints}, then
    \[
    \left(\psi_{\tau}|_{\tilde{E}}\right)^{\An}(\ord_F)=\psi_{\tau}^{\An}(\ord_F).
    \]
    We could easily reduce to the case $m=1$ and $\psi_{\tau}=\log|s|^2_{h_0}$ for a single holomorphic section $s$ of $L$ which does not vanish identically on $E$. In this case, the above equality is trivial.
\end{proof}

In particular, the trace operator is compatible with the partial order and the addition by a constant. 

As an interesting corollary, 
\begin{corollary}
    Assume the same assumptions as in \cref{thm:comp_traceop}. Suppose that $\phi\in \PSH(L^{\An})$ is a non-Archimedean metric represented by $\Gamma\in \PSH^{\NA}(X,\theta)$ and $E$ is an irreducible reduced subvariety of $X$. Then
    \[
    \phi(v_{E,\triv})=\sup\left\{\tau\in \mathbb{R}: \nu(\Gamma^{\theta+\omega}_{\tau},E)=0\right\}
    \]
    for any K\"ahler form $\omega$ on $X$.
\end{corollary}

\subsection{Functoriality}\label{subsec:funct}

Let $\pi\colon Y\rightarrow X$ be a proper bimeromorphic morphism from a reduced irreducible unibranch compact K\"ahler space of dimension $n$. By Zariski's main theorem, the obvious pull-back map induces a bijection 
\[
\pi^*\colon \PSH(X,\theta)\cn \PSH(Y,\pi^*\theta).
\]
In particular, it induces a canonical bijection
\[
\pi^*:\PSH^{\NA}(X,\theta)_{>0}\cn \PSH^{\NA}(Y,\pi^*\theta)_{>0}
\]
as follows: The image of $\Gamma\in \PSH^{\NA}(X,\theta)_{>0}$ is given by 
\[
\pi^*\Gamma\colon (-\infty,\Gamma_{\max})\rightarrow \PSH(Y,\pi^*\theta),\quad 
(\pi^*\Gamma)_{\tau}\coloneqq \pi^*\Gamma_{\tau}.
\]
This induces a canonical bijection
\[
\pi^*\colon \PSH^{\NA}(X,\theta)\cn \PSH^{\NA}(Y,\pi^*\theta).
\]
The injectivity is trivial and the surjectivity follows from \cref{lma:canonicalident}.

Note that there is a canonical identification $Y^{\Div}\cn X^{\Div}$, so that we can identify $(\pi^*\Gamma)^{\An}$ and $\Gamma^{\An}$.

One trivially verifies that all the previous operations are functorial with respect to $\pi^*$.

\section{Comparison with Boucksom--Jonsson's theory}
Let $X$ be a reduced irreducible unibranch projective variety and $L$ be a pseudoeffective line bundle on $X$ with $\xi=c_1(L)$. Fix a smooth Hermitian metric $h_0$ on $L$ with $\theta=c_1(L,h_0)$.

Recall the following basic result:
\begin{theorem}[{\cite[Theorem~4.22]{BJ18b}}]\label{thm:pshdetbydiv}
Let $\phi,\psi \in\PSH(L^{\An})$. Assume that $\phi\leq \psi$ on $X^{\Div}$, then the same holds on $X^{\An}$.
\end{theorem}
In particular, as we have computed the effect of our operations on $X^{\Div}$, their behaviours on $X^{\An}$ are in fact completely determined.
\begin{corollary}
    The order relation,  addition by a constant, maximum, decreasing limit, rescaling and addition correspond to the corresponding operations in Boucksom--Jonsson's theory. Assume that the envelope conjecture holds for $X$, then the same holds for the supremum operation.
\end{corollary}
Briefly speaking, we have established the correspondence between the non-Archimedean pluripotential theory and the convex geometry of complex potentials as in \cref{tab:corr}.

Finally, we have the following comparison result between our theory of non-Archimedean potentials and Boucksom--Jonsson's.

\begin{corollary}\label{cor:BJconj}
    The following are equivalent:
    \begin{enumerate}
        \item The envelope conjecture of Boucksom--Jonsson holds for $(X,L)$;
        \item The comparison map $\Gamma\mapsto \Gamma^{\An}$ defined in \cref{def:napotentialdaleth} admits an extension as a bijection $\PSH^{\NA}(X,\theta)\cn \PSH(L^{\An})$.
    \end{enumerate}
\end{corollary}
\begin{proof}
    By \cref{thm:pshdetbydiv}, if the map in (2) exists, it is necessarily unique.

    Let $\pi\colon Y\rightarrow X$ be a projective resolution. Consider the following diagram:
    \[
\begin{tikzcd}
\PSH^{\NA}(X,\theta) \arrow[r,"{\An}", dotted] \arrow[d,"\pi^*"] & \PSH(L^{\An}) \arrow[d,"\pi^*"] \\
\PSH^{\NA}(Y,\pi^*\theta) \arrow[r,"{\An}"]           & \PSH(\pi^*L^{\An}).          
\end{tikzcd}
\]
Clearly, (2) means that we can complete the dotted map as a bijection so that the diagram is commutative. We already know that the left-vertical map is a bijection, see \cref{subsec:funct}. The lower-horizontal map is also a bijection by \cref{thm:DXZ}. 

Now assume (1), then the right-vertical map is also a bijection by \cite[Lemma~5.13]{BJ18b}. Therefore, (2) follows. 

Conversely, assume that (2) holds. Take an increasing net $\phi_i\in \PSH^{\NA}(L^{\An})$ with $\phi_i\leq 0$. Recall that the envelope conjecture means that $\phi_i$ converges to some $\phi\in \PSH^{\NA}(L^{\An})$ with respect to the pointwise convergence topology on $X^{\Div}$. We write $\phi_i=\Gamma^{i,\An}$ for some $\Gamma^i\in \PSH^{\NA}(X,\theta)$. It suffices to take $\phi=(\sups_i \Gamma^i)^{\An} $ and apply \cref{thm:enve}.
\end{proof}

\bibliographystyle{amsalpha}
\bibliography{ComplexGeometry}

\end{document}